\documentclass[letterpaper, 10 pt, conference]{ieeeconf}  

\IEEEoverridecommandlockouts                              
\usepackage{graphics} 
\usepackage{epsfig} 
\usepackage{mathptmx} 
\usepackage{times} 
\usepackage{amsmath} 
\usepackage{amssymb}  

\usepackage{amsthm}
\usepackage{tikz}
\usepackage{subfig}
\usetikzlibrary{shapes,arrows,automata}
\newtheorem{assum}{Assumption}
\newtheorem{theorem}{Theorem}
\newtheorem{lemma}{Lemma}
\newtheorem{remark}{Remark}
\usepackage{epstopdf}
\newcommand\copyrighttext{%
  \centering \footnotesize Preprint submitted to the 60th IEEE Conference on Decision and Control
(CDC) 2021.}
\newcommand\copyrightnotice{%
\begin{tikzpicture}[remember picture,overlay]
\node[anchor=south,yshift=10pt] at (current page.south) {\fbox{\parbox{\dimexpr\textwidth-\fboxsep-\fboxrule\relax}{\copyrighttext}}};
\end{tikzpicture}%
}
\title{\LARGE \bf
Robustness of a feedback optimization scheme with application to bioprocess manufacturing
}

\author{ \parbox{1.3 in}{\centering Mirko Pasquini $^{a,b,*}$,\\
        \thanks{\noindent{$^a$} Division of Decision and Control, KTH Royal Institute of Technology, Sweden\\
        $^b$ Centre for Advanced Bioproduction, KTH Royal Institute of Technology, Sweden\\ 
        $^c$ Department of Industrial Biotechnology, KTH Royal Institute of Technology, Sweden\\
        $^*$ Corresponding author\\
        E-mail:
        {\tt\small \{pasqu, hjalmars, kcolin\} @kth.se\\ veronique.chotteau@biotech.kth.se}}}
        \hspace*{ 0.15 in}

        \parbox{1.2 in}{ \centering K\'evin Colin $^{a,b}$,
         \thanks{}}
        \parbox{1.8 in}{ \centering V\'eronique Chotteau $^{b,c}$  
        \thanks{This work has been supported by the VINNOVA Competence Centre AdBIOPRO, contract 2016-05181, and by the Swedish Research Council through the research environment NewLEADS - New Directions in Learning Dynamical Systems, contract 2016-06079.}}
        \parbox{1.8 in}{ \centering H{\aa}kan Hjalmarsson $^{a,b}$,
        }
        \hspace*{ 0.2 in}
        }

\begin{document}

\maketitle
\thispagestyle{empty}
\pagestyle{empty}

\copyrightnotice{}

\begin{abstract}

In this work the robustness of a feedback optimization scheme is discussed. Previously known results in literature, on the convergence to local optima of the optimization problem of interest, are extended to the case where the sensitivities of the steady-state input-output map of the plant present bounded uncertainties. The application of the scheme to a biological setting, with the goal of maximizing the concentration of products of interest in a bioreactor, under a continuous perfusion framework, is suggested and the potential of the approach is exposed by means of a simple synthetic example. 

\end{abstract}

\section{Introduction}
\textcolor{black}{
Feedback optimization is a control technique whose goal is to drive the steady-state output of a given plant, to local optima of an objective function, by taking into account output measurements of the physical system (see  \cite{Jokic2009,Bernstein2019} and references therein). This technique is gaining increasing interest in the control community due to its potential application in many practical fields, in particular recently focusing to power systems optimization \cite{Hauswirth2017, DallAnese2016, Ortmann2020}.
In \cite{Haberle2020} the authors described a feedback optimization framework in which the input is discretely updated, based on a projected gradient descent scheme. Constraints in the input variables are considered and output constraints are satisfied asymptotically, with a violation which can be quantified step-by-step. The ease of tuning and the capacity of handling non-convex problems, makes this approach extremely appealing. \\ It is well known, and remarked in the aforementioned paper, that a feedback optimization scheme would be inherently more robust than an open loop one, and it is conjectured that the approach would work as well in the presence of uncertainties in the input-output map sensitivities. Indeed robustness of the method has been experimentally validated in \cite{Ortmann2020}, while in \cite{Colombino2019} the authors worked in the directions of theoretical robustness guarantees for such a scheme. Here the authors, by considering a constant approximation of the input-output map sensitivities, gave LMI conditions for convergence to an online approximate solution, namely the solution that would be obtained if the approximate was true. However, differently from the case of \cite{Haberle2020}, the authors did not consider constraints on the output, simply enforcing soft constraints, by adding penalty terms in the objective function. \\ \\
Our goal here is twofold. First of all we would like to formalize a robust convergence result for the framework presented in \cite{Haberle2020}, introducing uncertainties in the knowledge of the Jacobian of the steady-state input-output map, but still considering output constraints, differently from what has been in done in \cite{Colombino2019}. Our main result (Theorem \ref{thm:convergence_uncertain setting}) will guarantee, under appropriate conditions, the convergence to a local optimum of the constrained optimization, if the step of the control algorithm is chosen small enough, contrary to \cite{Colombino2019} where the convergence is guaranteed only to an approximate solution whose distance from the real one depends on the goodness of the approximation.\\ \\
The second goal we have is, through a synthetic example, to propose the possibility of using such a feedback optimization scheme in a bioprocessing framework. A crucial aspect of bioprocessing production is the optimization of feed-flow medium composition, in order to maximize products of interest yield \cite{Chotteau2015, Singh2017}. For processes where mammalian cells are used for the production of the molecule of interest, such systems are highly complex with large uncertainties in the used models, making model based optimization prone to provide sub-optimal solutions. For this reason feedback optimization, where process measurements are used directly, seems to be an interesting alternative. However, as a model is still needed to be able to compute the Jacobian, the method needs to be robust to counter the modeling errors. \\
The problem we will consider is to maximize the concentration in the bioreactor at steady state, for a product of interest in a metabolic network, by optimizing the feed-flow medium composition in a continuous perfusion context i.e. the case in which medium is continuously added to the bioreactor and its content is continuously collected.\\ \\
The paper is structured as follows: in Section \ref{sec:feedback_optimization} the feedback optimization framework from \cite{Haberle2020} will be briefly recalled, together with their main result on the convergence of the control algorithm to local optima in the nominal case, while in Section \ref{sec:robustness_result} the main result of the paper, an extension of the convergence result in the presence of bounded uncertainties in the input-output map Jacobian, will be stated and its proof postponed in the Appendix. In Section \ref{sec:example} the problem of maximising the concentration of a product of interest, for a synthetic metabolic network in a continuous perfusion setting, is faced in order to propose feedback optimization as a valuable option in feed-flow medium optimization. Finally Section \ref{sec:conclusion} concludes the paper and presents some possible future extensions.  }
\subsection*{Preliminaries and notation}
We consider any vector $v \in \mathbb{R}^n$ to be a column, except for the gradient $\nabla f$ of a function $f: \Omega \subseteq \mathbb{R}^n \to \mathbb{R}$, which we consider to be a row vector. With $\nabla f$ we denote the gradient of the function $f$ if this is single-valued, while it refers to its Jacobian if it is vector-valued. For a vector $v$, with $||v||$ we indicate its norm and for a matrix $M$, with $||M||$ we indicate the matrix norm induced by the chosen vector norm, i.e.
\begin{equation}
    \label{eq:vector_induced_norm_matrix_norm}
    ||M|| := \sup \{{|| M v || \over ||v||},\forall v \ne 0\}
\end{equation}
With $\mathcal{S}_+^p$ we denote the set of positive semidefinite  symmetric matrix of dimension $p$, while $\mathbb{I}^p$ denotes the identity matrix of dimension $p$. Any matrix $G \in \mathcal{S}_+^p$ induces a vector 2-norm, formally:
\begin{equation*}
    || v ||_{G} = \sqrt{v^\top G v}
\end{equation*}
Ultimately, we say that a function $f:\Omega \subseteq \mathbb{R}^n \to \mathbb{R}^m$ is globally $L$-Lipschitz or simply $L$-Lipschitz on $\Omega$ if, for all $x$, $y \in \Omega$ it holds:
\begin{equation*}
    ||f(y) - f(x)|| \le L \cdot || y - x ||
\end{equation*}
\section{Feedback optimization}
\label{sec:feedback_optimization}
We recall here briefly the feedback optimization framework presented in \cite{Haberle2020}, to which the reader is referred for a more thorough description.\\
Feedback optimization is a control scheme that aims to steer a system's steady-state output to be the solution of a given optimization problem.\\ The plant's steady-state input-output map $h(u)$ is assumed to be unknown, but the plant's output can be measured. In \cite{Haberle2020}, the Jacobian $\nabla h$ of this map is assumed to be known exactly, an assumption that we will relax in Section \ref{sec:robustness_result}, by allowing the presence of bounded uncertainties.\\ \\
The considered problem is the following:
\begin{align}
    \label{eq:general_optimization_problem_feedback_opt}
    \min\limits_{u,\ y}& \  \Phi(u,y)\\
    \text{s.t.} \ \ &y = h(u) \nonumber\\
         &u \in \mathcal{U}, y \in \mathcal{Y} \nonumber
\end{align}
where $\Phi: \mathbb{R}^p \times \mathbb{R}^m \to \mathbb{R}$ is a continuosly differentiable function, $\mathcal{U} \subseteq \mathbb{R}^p$ and $\mathcal{Y} \subseteq \mathbb{R}^m$.
The following assumptions are made.
\begin{assum}
\label{assum:polyhedral_feasible_region}
    The feasible regions $\mathcal{U} \subseteq \mathbb{R}^p$ and $\mathcal{Y} \subseteq \mathbb{R}^m$ are described by:
    \begin{align*}
        \mathcal{U} = \{u \in \mathbb{R}^p \ | \ A u \le b, A \in \mathbb{R}^{q \times p}, b \in \mathbb{R}^q\}\\
        \mathcal{Y} = \{y \in \mathbb{R}^m\ |\ C y \le d, C \in \mathbb{R}^{l \times m}, d \in \mathbb{R}^l\}
    \end{align*}
    $\hfill \blacktriangleleft$
\end{assum}
\begin{assum}
\label{assum:feasibility_input_set_is_non_empty}
The following set:

\begin{equation*}
    \Tilde{\mathcal{U}} = \{u \in \mathbb{R}^p\ |\ A u \le b, Ch(u) \le d\} = \mathcal{U}\cap  h^{-1}(\mathcal{Y}) 
\end{equation*}
is non-empty.
    $\hfill \blacktriangleleft$
\end{assum}
The static input $u$ will evolve in a discrete fashion, according to the following rule:
\begin{align}
\label{eq:input_evolution_nominal}
    u^+ = u + \alpha  \hat{\sigma}_\alpha(u,y)\\ y = h(u) \nonumber
\end{align}
with:
\begin{align}
\label{eq:sigma_alpha_problem_nominal}
    \hat{\sigma}_\alpha(u,y) := &\arg\min\limits_{w \in \mathbb{R}^p} || w + G^{-1}(u) H(u)^\top \nabla \Phi(u,y)^\top||^2_{G(u)}\\
    s.t.  \quad & A(u + \alpha w) \le b  \nonumber\\
          \quad & C(y  +\alpha \nabla h(u) w) \le d \nonumber
\end{align}
where $G(u) : \mathbb{R}^p \to \mathcal{S}_+^p$ is a metric defined on $\mathbb{R}^p$, i.e. a map that to any point $u \in \mathbb{R}^p$ associates a  semidefinite positive matrix, $H(u)^T = \begin{bmatrix} \mathbb{I}^p & \nabla h(u)^\top \end{bmatrix}$ and $A$, $b$, $C$ and $d$ are specified in Assumption \ref{assum:polyhedral_feasible_region}. Given the discrete nature of the algorithm and the fact that $h$ is the plant's steady-state input-output map, it is implicitly assumed that the system dynamics will reach a steady-state before the next input $u^+$ will be applied.\\ \\
It is noted that, because the map $h$ is unknown, exact constraints on the output $y$ cannot be enforced, hence  violations of output constraints are admitted, but these are bounded and can be quantified --as showed in the proof of \cite[Lemma 4]{Haberle2020} -- and it is assured that the solution will converge to a feasible one.\\ The following assumptions are made to ensure the convergence of the algorithm to an optimality point of \eqref{eq:general_optimization_problem_feedback_opt}. 

\begin{assum}
\label{assum:LICQ_qualification}
For all $u \in \mathcal{U}$ the feasible set:
\begin{equation*}
    \Tilde{\mathcal{U}}_u := \{w | A(u + \alpha w) \le b, \ C(y + \alpha \nabla h(u) w) \le d\}
\end{equation*}
is non-empty and satisfies Linear Independence Constraint Qualification (LICQ) (see \cite{Haberle2020} for the formal definition), for all $w \in \Tilde{\mathcal{U}}_u$.     $\hfill \blacktriangleleft$
\end{assum}

\begin{assum}
    \label{assum:U_is_compact}
    For \eqref{eq:general_optimization_problem_feedback_opt}, $\mathcal{U}$ is compact.     $\hfill \blacktriangleleft$
\end{assum}
With the above assumptions, the following Theorem holds:
\begin{theorem}
\label{thm:convergence_nominal}
Let Assumptions 2 -- 4 hold. Let $\nabla \Phi$ and $\nabla h$ be globally Lipschitz.
Then $\exists \alpha^* > 0$ s.t. for all $0 < \alpha < \alpha^*$ s.t. given any trajectory of $u$, for any $u^0 \in  \mathcal{U}$, obtained through the updating rule \eqref{eq:input_evolution_nominal} -- \eqref{eq:sigma_alpha_problem_nominal}, then $y  = h(u)$ converges to the set of first-order optimality points of \eqref{eq:general_optimization_problem_feedback_opt}
\end{theorem}

\section{Main contribution - Robustness result}
\label{sec:robustness_result}
In this section Theorem \ref{thm:convergence_nominal} is extended to the case where the Jacobian of $h(u)$ is uncertain.\\ \\
Let the Jacobian of the  map $h(u)$, be known up to an additive uncertainty, and be denoted with $\widehat{\nabla h(u)}$ i.e.:
\begin{equation}
    \label{eq:known_Jacobian}
    \widehat{\nabla h(u)} := \nabla h(u) + \delta(u)
\end{equation}
In the following we will assume that the matrix norm of $\delta(u)$ is bounded. In this case the update rule will remain as in Eq. \eqref{eq:input_evolution_nominal}, while the optimization problem solved to find $\hat{\sigma}_{\alpha}$ will now become:
\begin{align}
\label{eq:sigma_alpha_problem_uncertain_formulation}
    \hat{\sigma_\alpha}(u,y) := &\arg\min\limits_{w \in \mathbb{R}^p} || w + G^{-1}(u)\hat{H}(u)^\top \nabla\Phi(u,y)^\top||^2_{G(u)}\\
    s.t.  \quad & A(u + \alpha w) \le b  \nonumber\\
          \quad & C(y  +\alpha \nabla h(u) w + \alpha \delta(u) w) \le d \nonumber
\end{align}
where, in this case, 
\begin{equation*}
    \hat{H}(u)^\top = \begin{bmatrix} \mathbb{I}^p &  \widehat{\nabla h(u)}^\top \end{bmatrix}
\end{equation*}

Before proceeding, Assumption \ref{assum:LICQ_qualification} is substituted with the following.
\begin{assum}
\label{assum:LICQ_qualification_mod}
For all $u \in \mathcal{U}$ the feasible set:
\begin{equation*}
    \Tilde{\mathcal{U}}_u := \{w | A(u + \alpha w) \le b, \ C(y + \alpha \nabla h(u) w + \alpha \delta(u) w) \le d\}
\end{equation*}
is non-empty and satisfies LICQ for all $w \in \Tilde{\mathcal{U}}_u$.     $\hfill \blacktriangleleft$
\end{assum}
Part of the proof for the Theorem below --which is moved in the Appendix for ease of reading -- is based on steps similar to the ones in the proof of what has been reported here as Theorem \ref{thm:convergence_nominal}, presented in \cite{Haberle2020}. In particular the first half of the proof, up to Lemma \ref{lem:Lyapunov_difference}, follows tightly the steps of \cite{Haberle2020} -- except for the presence of the uncertainty $\delta(u)$, and the consequent adaptations -- while the proof of Lemma \ref{lem:Lyapunov_difference}, which is central to the main result, is approached in an almost completely different way, due to the presence of mixed terms dependent on the uncertainty $\delta(u)$. All the main steps of the proof are reported, for clarity of explanation, although some intermediate steps or some passages identical to the nominal case are here skipped.\\
We remark that the spirit of the result is the same as Theorem \ref{thm:convergence_nominal}, i.e. that by choosing an $\alpha$ small enough, the feedback optimization control algorithm will drive the output of the system to solve the optimization problem \eqref{eq:general_optimization_problem_feedback_opt}, despite the presence of an uncertainty in the knowledge of the steady-state input-output map sensitivities.
\begin{theorem}
\label{thm:convergence_uncertain setting}
Let Assumptions \ref{assum:feasibility_input_set_is_non_empty}, \ref{assum:U_is_compact} and \ref{assum:LICQ_qualification_mod} hold. Let $\nabla \Phi$ and $\nabla h$ be globally Lipschitz. Let $\widehat{\nabla h(u)}$ be defined as \eqref{eq:known_Jacobian} and let $\delta(u)$ be continuous and its norm bounded i.e.:
\begin{equation}
\label{eq:bounded_delta_M_delta_condition}
    ||\delta(u)|| \le M_\delta
\end{equation}
Assume further that:
\begin{equation}
\label{eq:bounded_nabla_y_Phi_M_nabla_condition}
    ||\nabla_y \Phi(u,y)|_{y = h(u)}|| \le M_\nabla
\end{equation}
where $\nabla_y \Phi(u,y)|_{y = h(u)}$ is the sub-vector of the gradient of $\Phi(u,y)$ relative to $y$, evaluated in $y = h(u)$.\\
Then $\exists \alpha^* > 0$ s.t. for all $0 < \alpha < \alpha^*$, any trajectory of $u$, obtained through the updating rules \eqref{eq:input_evolution_nominal}  and \eqref{eq:sigma_alpha_problem_uncertain_formulation}, for any $u^0 \in  \mathcal{U}$, is such that $y  = h(u)$ converges to the set of first-order optimality points of \eqref{eq:general_optimization_problem_feedback_opt} $\hfill \blacktriangleleft$\\
\end{theorem}

\section{Bioprocessing example}
\label{sec:example}
We will consider a continuous perfusion bioprocess example for which a certain product concentration in the bioreactor,  at steady state, has to be maximised. We consider the metabolic network in Fig. \ref{fig:fict_metabolic_net}.
\begin{figure}[h!]
    \centering
    \includegraphics[width=0.45\textwidth]{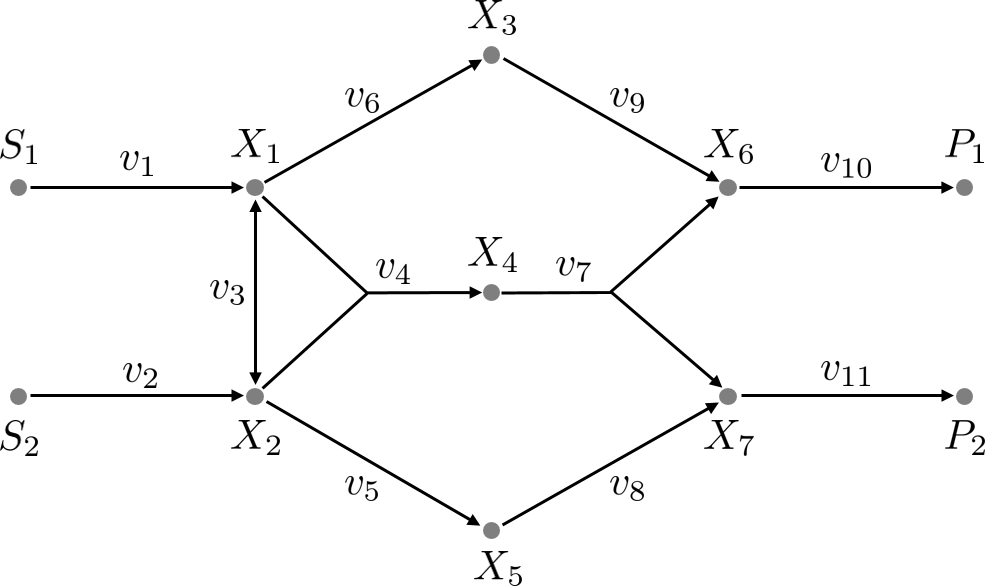}
    \caption{Example metabolic network.}
    \label{fig:fict_metabolic_net}
\end{figure}
where $S_1$ and $S_2$ are substrates, $P_1$ and $P_2$ are products and $X_i$, $i = 1, \dots, 7$ are internal metabolites. The arrows represent metabolic reactions that happen with rates $v_i$, generally dependent on both reactant and product concentrations. All the reactions in the example network are irreversible, with the only exception of the one whose rate is $v_3$.\\
After some biologically reasonable assumptions, we can rewrite the vector $v$ of fluxes as the combination of a set of basis vectors known as Elementary Flux Modes:
\begin{equation}
\label{eq:EFM_flux_rep}
    v = E \eta
\end{equation}
The interested reader is referred to \cite{Schuster1994,Gagneur2004,Bastin2008} for a more thorough discussion on the topic. It is remarked that the matrix $E$ in \eqref{eq:EFM_flux_rep} can be evaluated using available software tools (see METATOOL \cite{METATOOL}) and depends only on the network topology and the reversibility of the reactions. The matrix $E \in \mathbb{R}^{10 \times 4}$ is not reported here for the sake of space, but we collected all the numerical details of this section in \cite{arxivNumerical} to allow reproducibility of the results.\\
In a continuous perfusion setting, substrates can be continuously added in the reactor, while cell broth --i.e. all that developed in the bioreactor, including products, substrates and biomass in general -- can be continuously removed \cite{Chotteau2015}. By considering only external metabolites, the ones that could actually be measured, the expression \eqref{eq:EFM_flux_rep} and the fact that a continuous perfusion setting is considered, when the system will reach a steady-state the following  mass-balance equation will hold:
\begin{equation}
    \label{eq:mass_balance_eq}
    0 = N A_{ext} E \eta - F_1 c_{ext} + F_2 S_{in}
\end{equation}
where $N$ is the biomass at steady-state, $c_{ext} = \begin{bmatrix} s_1 &s_2 & p_1 & p_2 \end{bmatrix}^\top$ is the vector of external metabolites concentrations, $A_{ext}$ is the stoichiometric matrix relative to the external metabolites, $S_{in}$ is the vector of substrate concentrations in the feed-flow medium, while:
\begin{align}
\label{eq:F_matrices}
    F_1 &= \begin{bmatrix} F &0 &0 &0 \\ 0 &F &0 &0\\ 0 &0 &F &0\\ 0 &0 &0 &F \end{bmatrix}\\
    F_2 &= \begin{bmatrix} F &0 \\ 0 &F\\ 0 &0\\ 0 &0  \end{bmatrix}
\end{align}
where $F$ is defined as flow rate per bioreactor volume.
We assume standard Monod-type kinetics expressions for the components of $\eta$, dependent only on the substrates concentration, i.e.:
\begin{align}
    \label{eq:monod_expressions}
    \eta_j = \eta_{j,\max} {s_1 \over {s_1 + k^a_{1,j}}} \cdot {1 \over {1 + k^i_{1,j} s_1}} \cdot {s_2 \over {s_2 + k^a_{2,j}}} \cdot {1 \over {1 + k^i_{2,j} s_2}}
\end{align}

We will apply the Feedback Optimization scheme \eqref{eq:input_evolution_nominal} -- \eqref{eq:sigma_alpha_problem_nominal}, by considering:
\begin{align*}
    y &\leftarrow c_{ext}\\
    u &\leftarrow S_{in}\\
    \Phi &\leftarrow \lambda^\top y\\
    G(u) &\leftarrow \mathbb{I}^2
\end{align*}
We choose to maximize the concentration of the product $P_2$ so that, given the expression of $c_{ext}$, one have:
\begin{equation*}
    \lambda^\top = \begin{bmatrix} 0 &0 &0 &1 \end{bmatrix}
\end{equation*}
The set of feasible inputs is defined by a set of lower and upper bounds on the inflow medium concentrations i.e.:
\begin{align*}
    0 \le \underline{S_{in}^1} \le S_{in}^1 \le \overline{S_{in}^1}\\
    0 \le \underline{S_{in}^2} \le S_{in}^2 \le \overline{S_{in}^2}
\end{align*}
In the same way, constraints on the output variables i.e. the external metabolites concentration at steady-state in the bioreactor, are given in terms of lower and upper bounds:
\begin{align*}
    0 \le \underline{s_i} \le s_i \le \overline{s_i}, \quad i \in \{1,2\}\\
    0 \le \underline{p_i} \le p_i \le \overline{p_i}, \quad i \in \{1,2\}
\end{align*}
Given the model \eqref{eq:mass_balance_eq}, we can obtain the Jacobian of the input-output steady-state map $y = h(u)$ by applying the implicit function theorem, which leads to:
\begin{equation}
    \label{eq:input_output_map_Jacobian}
    \nabla h(u) = [N A_{ext} E \nabla \eta(y) - F_1]_{y = h(u)}^{-1} F_2
\end{equation}
Being extremely difficult to model exactly the kinetic expression of $\eta$ we will assume to know an approximation $\hat{\eta}$ of it, which will lead to:
\begin{equation}
    \label{eq:input_output_map_Jacobian}
    \widehat{\nabla h}(u) = [N A_{ext} E \nabla \hat{\eta}(y) - F_1]_{y = h(u)}^{-1} F_2
\end{equation}
Following the notation of Section \ref{sec:robustness_result} we have:
\begin{equation*}
    \delta(u) = \widehat{\nabla h}(u) - {\nabla h}(u)
\end{equation*}
implying:
\begin{align*}
    ||\delta(u)|| \le &||[N A_{ext} E \nabla \hat{\eta}(y) - F_1]_{y = h(u)}^{-1}||\cdot||F_2|| \\ + &||[N A_{ext} E \nabla \eta(y) - F_1]_{y = h(u)}^{-1}|| \cdot||F_2|| 
\end{align*}
By assuming that the matrices $N A_{ext} E \nabla \hat{\eta}(y) - F_1$ and $N A_{ext} E \nabla \eta(y) - F_1$ are non-singular for any $u \in \mathcal{U}$ and considering the matrix norm induced by the $2$-norm of vectors and the fact that $\mathcal{U}$ is compact, it follows that $\delta(u)$ is norm-bounded and continuous.\\
We will now examine two cases: the first one in which the kinetics will be affected by uncertainties in the parameters of the Monod functions, while the second case considers more systematic errors in the modelling, which in this particular case will be simulated by removing one or more columns of the matrix $E$, as well as the corresponding components of the vector $\eta$, in \eqref{eq:EFM_flux_rep}.
\subsection{Parametric uncertainties}
A random perturbation of the kinetic parameters in a $50\%$ range is considered, obtained by uniformly sampling any parameter $\theta$ in the range $[0.5 \bar{\theta}, 1.5 \bar{\theta}]$ where $\bar{\theta}$ is the nominal value of the considered parameter. Six realizations of the uncertainty are simulated. The obtained results are shown in Figure \ref{fig:uncertain_parameters_test} and show clearly how, with a low $\alpha$ (in this case $\alpha = 0.0015$), convergence of the concentration of $P_2$ to the optimum value can be obtained, despite the large parametric uncertainty.
\begin{figure}[h!]
    \centering
    \includegraphics[width=0.45\textwidth]{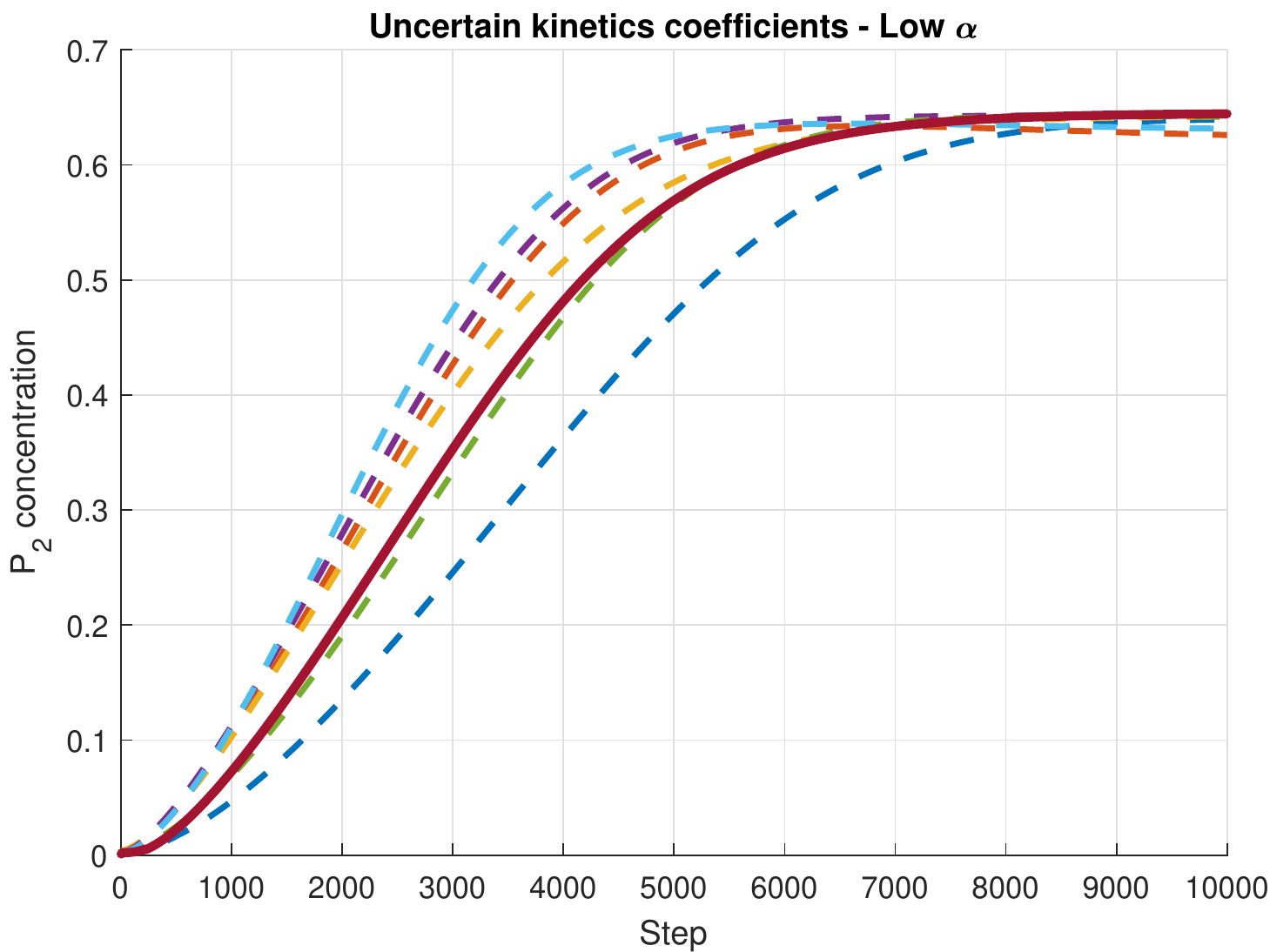}
    \caption{$P_2$ concentration evolution, with a low $\alpha$, for perturbed kinetics parameters. The solid line is the curve obtained in the nominal case, at the chosen $\alpha$, while the dashed lines are the curves obtained in six different realizations of the uncertainty.}
    \label{fig:uncertain_parameters_test}
\end{figure}
\subsection{Missing kinetics components}
In this second case we consider a more systematic error in the kinetics modelling. In particular we consider six different cases, in five of which one or two columns of $E$ are removed, simulating the event where part of the kinetics has not been modelled. As in the previous case, in Figure \ref{fig:missing_kinetics_test} it is possible to observe convergence of the concentration of $P_2$ to the real optimum value when the step of the Feedback Optimization algorithm is chosen to be small enough (in this case $\alpha = 0.0015$). 
\begin{figure}[h!]
    \centering
    \includegraphics[width=0.45\textwidth]{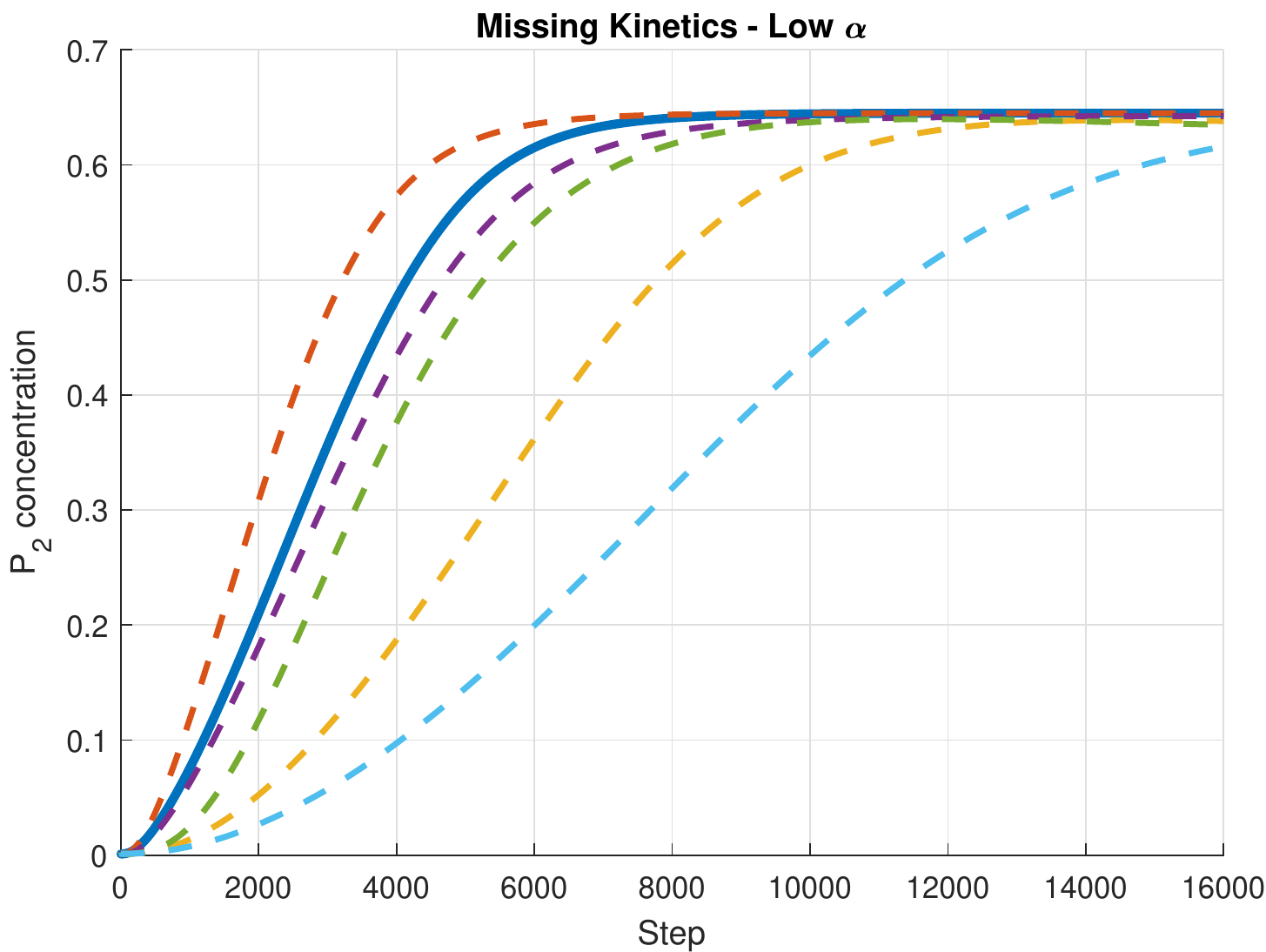}
    \caption{$P_2$ concentration evolution, with a low $\alpha$, when parts of the kinetics are not modelled. The solid line is the curve obtained in the nominal case, at the chosen $\alpha$, while the dashed lines are the curves obtained in the five cases where one or two columns of $E$ are removed.}
    \label{fig:missing_kinetics_test}
\end{figure}
\section{Conclusions and further developments}
\label{sec:conclusion}
Feedback optimization is a powerful control technique, when a system's steady-state output should be driven to optimize an objective function. In this work we considered a framework, available in literature, in which the steady-state input-output map is unknown and input and output constraints are included, and we extended the convergence result of the algorithm to the case where bounded uncertainties are present in the map Jacobian. A bioprocessing example has been discussed, to show the potential applicability of the method, to the problem of optimizing the feed-flow medium composition for a bioreactor in continuous perfusion mode, to maximise the yield of molecules of interest. The high complexity of these systems, together with the difficulties of defining an accurate model, inspire the use of a feedback scheme in which output measurements are directly used and a degree of robustness can be expected. \\
There are possible future developments of this work. In particular, while the class of admitted uncertainty is very large, asking for their continuity with respect to the input will automatically exclude the possibility of considering random noise or to use non-smooth approximation of the Jacobian (e.g. the sign of the Jacobian). Moreover we recognize that condition \eqref{eq:bounded_nabla_y_Phi_M_nabla_condition} can be limiting and possibly difficult to verify in general, hence future works should address how this can be relaxed. Another main issue might be the convergence speed of the algorithm. In fact, while choosing a low step $\alpha$ guarantees robustness of the approach, it also worsen the convergence speed of the algorithm. This is a problem in particular in the bioprocessing area, as any single experiment can be expensive, both in terms of resources and time. Possible heuristics where the step-size is dynamically changed can be considered, and should be further investigated. Finally the approach described in Section \ref{sec:example} should be experimentally validated. 
\section*{Appendix - Proof of Theorem \ref{thm:convergence_uncertain setting}}
\begin{proof}
We start by rewriting \eqref{eq:sigma_alpha_problem_uncertain_formulation}, as:
\begin{align}
\label{eq:sigma_alpha_problem_uncertain_expanded}
    \hat{\sigma_\alpha}(u,y) := &\arg\min\limits_{w \in \mathbb{R}^p} || w + G^{-1}(u)(H(u)^\top \nabla \Phi(u,y)^\top + \\ &\delta(u)^\top \nabla_y \Phi(u,y)^\top)||^2_{G(u)}\\
    s.t.  \quad & A(u + \alpha w) \le b  \nonumber\\
          \quad & C(y  +\alpha \nabla h(u) w + \alpha \delta(u) w) \le d \nonumber
\end{align}
By closing the feedback loop (i.e. by considering $y = h(u)$) we have $\tilde\Phi(u) = \Phi(u,h(u))$, and for the rest of the proof we will have:
\begin{align}
    \label{eq:input_evolution_feedback}
    u^+ = &u + \alpha \sigma_\alpha(u) =: T(u)\\
    \label{eq:sigma_alpha_uncertain_feedback}
    \sigma_\alpha(u) = & \arg\min\limits_{w \in \mathbb{R}^p} || w + G^{-1}(u)[\nabla\tilde{\Phi}(u)^\top + \\ &\delta(u)\nabla_y\Phi(u,y)|_{y = h(u)}]||_{G(u)}^2 \\ \label{eq:feedback_constraint_on_input} \text{s.t.} \quad & A(u + \alpha w) \le b  \\ \label{eq:feedback_constraint_on_output}
          \quad & C(h(u)  +\alpha \nabla h(u) w + \alpha \delta(u) w) \le d  
\end{align}

We notice that, given that $\nabla h$ and $\nabla\Phi$ are Lipschitz and that $\mathcal{U}$ is compact, the term $\nabla\tilde{\Phi} $ is Lipschitz.\\  
To apply the invariance principle \cite{Lasalle1976} we first need to prove that the map $u^+ = T(u)$ is invariant with respect to $\mathcal{U}$ and that it is continuous in $u$.\\
These two facts can be proven by following the same steps as in \cite{Haberle2020}, as \eqref{eq:feedback_constraint_on_input} holds and the dependence from $w$, of the functions involved in \eqref{eq:sigma_alpha_uncertain_feedback}--\eqref{eq:feedback_constraint_on_output}, is the same in the objective function, and it is still linear in the constraints.\\
Before proceeding we recall the following Descent Lemma \cite{Bertsekas1999}, which will be instrumental in the following.
\begin{lemma}[Descent Lemma]
\label{lem:descent_lemma}
Given a continuously differentiable function $f : \mathbb{R}^p \to \mathbb{R}$ with L-Lipschitz derivative $\nabla f$, for all $x$, $z \in \mathbb{R}^p$
it holds that $f(z) \le f(x) + \nabla f(x)(z - x) + {L \over 2} || z - x ||^2$. $\hfill \blacktriangleleft$
\end{lemma}
We now prove the following Lemma, showing how the constraint violation on the output, committed at every step of the algorithm, can be bounded. 
\begin{lemma}
\label{lem:output_violation_bound}
 Let $l_i$ be the Lipschitz constant for $C_i \nabla h(u)$ , $\forall i \in 1,\ \dots, \ l$. Given the iteration \eqref{eq:input_evolution_feedback}, and any $u \in \mathcal{U}$ we have:
\begin{equation*}
    C_i h(u^+)  - d_i \le - \alpha C_i \delta(u) \sigma_\alpha(u) + {l_i\over 2}||\alpha \sigma_\alpha(u)||^2 
\end{equation*}
\end{lemma}
\begin{proof}
The proof of this fact is similar to \cite[Lemma 4]{Haberle2020}. First of all we point out that $C_i \nabla h(u)$ is Lipschitz, as $\nabla h(u)$ is Lipschitz and $C_i$ is a constant vector.\\
Using the Descent Lemma  we have:
\begin{equation*}
    C_i h(u^+) - C_i h(u) \le C_i \nabla h(u) (\alpha \sigma_{\alpha}(u)) + {l_i\over 2}||\alpha \sigma_\alpha(u)||^2 
\end{equation*}
Given that \eqref{eq:feedback_constraint_on_output} should be satisfied by $ w = \sigma_\alpha(u)$, it follows:
\begin{equation*}
    C_i h(u^+)  - d_i \le - \alpha C_i \delta(u) \sigma_\alpha(u) + {l_i\over 2}||\alpha \sigma_\alpha(u)||^2 
\end{equation*} $\hfill \qed$
\end{proof}

The main idea of the proof is to find a Lyapunov function $V$ and show that this is non-increasing along the trajectories of $u$, for any initial condition $u_0$, so that the Invariance Principle \cite{Lasalle1976} can be applied. We will consider the same Lyapunov function as in \cite{Haberle2020}:
\begin{equation}
\label{eq:Lyapunov_function_definition}
V(u) = \tilde{\Phi}(u) + \xi [\sum\limits_{i=1}^l \max\{0,C_i h(u) - d_i\}] 
\end{equation}
where $\xi \in \mathbb{R}$ has the property that:
\begin{equation}
\label{eq:xi_sup_property}
    \xi \ge \sup\limits_{\substack{u \in \mathcal{U} \\ i = 1, \dots, l}} \{\mu_i^*(u)\}
\end{equation}
with $\mu_i^*(u)$ being the Lagrange multiplier with respect to the $i$-th constraint in \eqref{eq:feedback_constraint_on_output}. It is remarked that such $\xi$ exists due to \cite[Theorem 2]{Haberle2020}.
To show the monotonicity property of the Lyapunov function with respect to any given trajectory $\bar{u} :=\{u_0, u_1, \dots\}$, we will rewrite problem \eqref{eq:sigma_alpha_uncertain_feedback} -- \eqref{eq:feedback_constraint_on_output} in an equivalent form, as follows:
\begin{align}
   \label{eq:transformed_sigma_alpha_problem}
   \sigma_\alpha(u) = & \arg\min\limits_{w \in \mathbb{R}^p} \alpha( {1 \over 2}w^\top G(u) w + \nabla \tilde \Phi(u)^\top w  + \\ &\delta(u) \nabla_y\Phi(u,y)|_{y = h(u)}^\top w) \\ \label{eq:feedback_constraint_on_input_transformed} \text{s.t.} \quad & \alpha A w \le b - A u  \\ \label{eq:feedback_constraint_on_output_transformed}
   \quad &  \alpha C [\nabla h(u) + \delta(u)] w  \le d - C(h(u) 
\end{align}
where all the terms independent of $w$ in the objective have been removed and the objective function has been multiplied by the positive scalar $\alpha \over 2$.
Before stating and proving the Lemma about the non-increasing nature of $V$, we recollect the following KKT conditions, where $\nu$ and $\mu$ are the Lagrange multipliers of the constraints \eqref{eq:feedback_constraint_on_input_transformed} and \eqref{eq:feedback_constraint_on_output_transformed} respectively:
\begin{itemize}
    \item \textbf{Stationarity}: 
    \begin{align}
    \label{eq:stationarity}
        \alpha w^\top G(u) + \alpha \nabla \tilde \Phi(u) + \alpha \delta(u)\nabla_y \Phi(u,y)|_{y = h(u)} + \\ \alpha \nu^\top A + \alpha \mu^\top[C(\nabla h(u) + \delta(u))] = 0 \nonumber
    \end{align}
    \item \textbf{Complementary Slackness}: 
    \begin{align}
        \label{eq:complementary_slackness}
        \nu_j(\alpha A_j w - b_j + A_j u) &= 0\\  j = 1, \dots, q& \nonumber\\ \mu_i ( \alpha C_i \nabla h(u) w + \alpha C_i \delta(u) w - d_i + C_i h(u))&=0 \nonumber\\ i = 1, \dots, l& \nonumber
    \end{align}
    \item \textbf{Primal Feasibility}:
        \begin{align}
            \label{eq:primal_feasibility}
            \alpha A w &\le b - A u  \\ 
            \alpha C [\nabla h(u) + \delta(u)] w  &\le d - C(h(u)  \nonumber
        \end{align}
    \item \textbf{Dual Feasibility}:
        \begin{align}
            \label{eq:dual_feasibility}
            \nu_i \ge 0\\
            \mu_i \ge 0 \nonumber
        \end{align}
\end{itemize}
The above conditions will hold at the optimal points of \eqref{eq:transformed_sigma_alpha_problem} -- \eqref{eq:feedback_constraint_on_output_transformed}.
The following lemma can now be proved.
\begin{lemma}
\label{lem:Lyapunov_difference}
 Let Assumptions \ref{assum:feasibility_input_set_is_non_empty}, \ref{assum:U_is_compact} and \ref{assum:LICQ_qualification_mod} be satisfied. Let $V$ be the Lyapunov function defined in \eqref{eq:Lyapunov_function_definition}, where $\xi$ satisfies \eqref{eq:xi_sup_property}. Let $L$ be the Lipschitz constant for $\nabla \tilde{\Phi}(u)$  and $l_i$ be the Lipschitz constant for $C_i \nabla h(u)$, $\forall i \in 1,\ \dots, \ l$.   Given the update rule \eqref{eq:input_evolution_feedback}, $V(u^+) \le V(u)$ is satisfied for all $u \in \mathcal{U}$, if:
 \begin{equation}
     \label{eq:alpha_bound_theorem}
     0 < \alpha < \alpha^* 
 \end{equation}
 with
 \begin{equation*}
     \alpha^* = 2 {\lambda_{min}(G(u)) \over M_\delta \cdot M_{\nabla} + L + \xi \sum\limits_{i=1}^l{l_i\over 2 } + ||C_i|| M_\delta }
 \end{equation*}
\end{lemma}
\begin{proof}
First of all we write $V(u^+)-V(u)$ as:
\begin{align}
\label{eq:Lyapunov_difference_beginning}
    V(u^+)-V(u) = \tilde\Phi(u^+) - \tilde\Phi(u) + \\ \xi [\sum\limits_{i=1}^l (\max\{0,C_i h(u^+) - d_i\} -\max\{0,C_i h(u) - d_i\}) ] \nonumber
\end{align}
We first examine the term $\tilde\Phi(u^+) - \tilde\Phi(u)$. Being  $\nabla \tilde \Phi$ $L-\text{Lipschitz}$, from the Descent Lemma it follows:
\begin{align*}
    \tilde\Phi(u^+) &\le \tilde\Phi(u) + \nabla \tilde{\Phi}(u) (\alpha \sigma_\alpha(u)) + {L \over 2}||\alpha \sigma_\alpha(u)||^2
\end{align*}
From now on we write $ w = \sigma_\alpha(u)$, for ease of notation, as we are assuming that \eqref{eq:transformed_sigma_alpha_problem} -- \eqref{eq:feedback_constraint_on_output_transformed} is solved to update the input $u$, hence:
\begin{align}
    \label{eq:Phi_difference_Lyapunov_proof_bound}
    \tilde\Phi(u^+) - \tilde\Phi(u) \le \alpha \nabla \tilde{\Phi}(u) w + {L \over 2} \alpha^2 ||w||^2
\end{align}
We can further bound the term $\alpha \nabla \tilde{\Phi}(u) w$ by recurring to the above KKT conditions, in particular using Stationarity and Complementary Slackness, and we can write:
\begin{align}
\label{eq:alpha_nabla_tilde_Phi_u_w_bound}
    \alpha \nabla \tilde \Phi(u) w \le -\alpha w^\top G(u) w \\ - \alpha \delta(u) \nabla_y \Phi(u,y)|_{y = h(u)} w +   \sum\limits_{i=1}^l \mu_i \max\{0, C_i h(u) -d_i\} \nonumber
\end{align}
where the inequality follows from the fact that $u \in \mathcal{U}$ and that $r \le \max\{0,\ r\}$, for all $r \in \mathbb{R}$.\\ \\
Moreover, by applying Lemma \ref{lem:output_violation_bound}, it follows:
\begin{equation}
    \label{eq:output_violation_constraint_Lyapunov_proof}
    \max\{0,C_i h(u^+) - d_i\} \le \max\{0,{l_i\over 2}\alpha^2||w||^2 - \alpha C_i \delta(u) w\}
\end{equation}\\ \\
Combining \eqref{eq:Lyapunov_difference_beginning} --  \eqref{eq:output_violation_constraint_Lyapunov_proof}, and by the definition of $\xi$ it follows:
\begin{align}
    \label{eq:Lyapunov_difference_all_bound_inside}
    &V(u^+)-V(u) \le \\ &\le -\alpha w^\top G(u) w + {L \over 2} \alpha^2 ||w||^2 + \nonumber \\ &- \alpha \delta(u) \nabla_y \Phi(u,y)|_{y = h(u)} w  + \nonumber \\ &\xi \sum\limits_{i=1}^l\max\{0,{l_i\over 2}\alpha^2||w||^2 - \alpha C_i \delta(u) w\} \nonumber
\end{align}
To further bound the difference $V(u^+)-V(u)$, we now need to deviate from the proof of Theorem \ref{thm:convergence_nominal} in \cite{Haberle2020} as, differently from the nominal case, various mixed terms dependent on $\delta(u)$ appear.\\
We recall the following property, valid for any couple of vectors $u$, $v \in \mathbb{R}^n$:
\begin{equation}
    \label{eq:scalar_product_rewrite}
    u^\top v = ||{1\over2}u + v||^2 - {1\over4}||u||^2 - ||v||^2
\end{equation}
which allows us to write:
\begin{align}
    \label{alpha_delta_u_rewritten_with_norms}
    &\alpha C_i \delta(u) w = \\ &= ||{1\over 2} (\alpha C_i \delta(u))^\top + w||^2 - {1\over4}
    ||(\alpha C_i \delta(u))^\top||^2 - ||w||^2 \nonumber 
\end{align}
Moreover, from the triangle inequality of norms, the following inequality holds, for any couple of vectors $u$, $v \in \mathbb{R}^n$:
\begin{equation}
    \label{eq:chain_of_inequalities_from_triangle_one}
     ||u + v||^2 \ge ||u||^2 + ||v||^2 - 2||u||\cdot ||v|| 
\end{equation}
From \eqref{eq:chain_of_inequalities_from_triangle_one} it follows:
\begin{align*}
    -||{1\over 2} (\alpha C_i \delta(u))^\top + w||^2 \le \\ \le -||{1\over 2} (\alpha C_i \delta(u))^\top||^2 - ||w||^2 + {1\over 2}  ||(\alpha C_i \delta(u))^\top||\cdot||w||
\end{align*}
which used together with \eqref{alpha_delta_u_rewritten_with_norms} implies:
\begin{align}
\label{eq:bound_on_max_0_li_2_etc}
    &\max\{0,{l_i\over 2}\alpha^2||w||^2 - \alpha C_i \delta(u) w\} \le \\ \le &\max\{0,{l_i\over 2 }\alpha^2||w||^2 + {1\over 2}  ||(\alpha C_i \delta(u))^\top||\cdot||w||\} = \nonumber \\ = &{l_i\over 2 }\alpha^2||w||^2 + {1\over 2}  ||(\alpha C_i \delta(u))^\top||\cdot||w||\nonumber
\end{align}
where the last equality follows from the positivity of all the components in the second term of the max function.\\
Following a similar reasoning for the term $\alpha \delta(u) \nabla_y \Phi(u,y)|_{y = h(u)} w$, we get:
\begin{align}
    \label{eq:bound_on_alpha_delta_nabla_Phi}
    &-\alpha \delta(u) \nabla_y \Phi(u,y)|_{y = h(u)} w  \le \\ \le &{1\over2}||\alpha \delta(u) \nabla_y\Phi(u,y)|^\top_{y=h(u)}|| \cdot ||w||  \nonumber
\end{align}
By the definition of matrix norm induced by vector norm, it follows
\begin{equation}
    ||A x || \le ||A|| \cdot ||x||, \quad \forall x \in \mathbb{R}^n
\end{equation}
for any $ A \in \mathbb{R}^{m \times n}$, which in turn implies that:
\begin{align}
    \label{eq:matrix_norm_driven_bound_on_matrix_vector_product}
    ||\alpha \delta(u) \nabla_y\Phi(u,y)|^\top_{y=h(u)}|| &\le ||\alpha \delta(u)|| \cdot ||\nabla_y\Phi(u,y)|^\top_{y=h(u)}||\\ ||(C_i \delta(u))^\top|| &\le ||C_i^\top|| \cdot ||\delta(u)|| \nonumber
\end{align}
From \eqref{eq:bound_on_max_0_li_2_etc} -- \eqref{eq:bound_on_alpha_delta_nabla_Phi} and \eqref{eq:matrix_norm_driven_bound_on_matrix_vector_product}, we can rewrite \eqref{eq:Lyapunov_difference_all_bound_inside} as:
\begin{align}
\label{eq:Lyapunov_difference_almost_at_the_end_with_bounds}
    V(u^+)-V(u) \le \\ \le -\alpha \lambda_{min}(G(u)) \cdot ||w||^2 + {L \over 2} \alpha^2 ||w||^2 +\nonumber\\  {1\over2}\alpha^2|| \delta(u) || \cdot || \nabla_y\Phi(u,y)|_{y=h(u)}|| \cdot ||w|| +\nonumber\\ \xi \sum\limits_{i=1}^l{l_i\over 2 }\alpha^2||w||^2 + {1\over 2}  \alpha^2 ||C_i|| \cdot ||\delta(u)||\cdot||w|| \nonumber
\end{align}
The terms $||\delta(u)||$ and $|| \nabla_y\Phi(u,y)|_{y=h(u)}||$ are bounded by hypothesis,
so \eqref{eq:Lyapunov_difference_almost_at_the_end_with_bounds} can be rewritten as:
\begin{align}
\label{eq:final_Lyapunov}
    V(u^+)-V(u) \le \\ \le -\alpha \lambda_{min}(G(u)) \cdot ||w||^2  + {L \over 2} \alpha^2 ||w||^2 + \\ {1\over2}\alpha^2 M_\delta \cdot M_{\nabla} \cdot ||w||    + {1\over 2}  \alpha^2 ||C_i|| M_\delta \cdot||w|| +\\ \xi \sum\limits_{i=1}^l{l_i\over 2 }\alpha^2||w||^2 \nonumber
\end{align}
By considering the three cases where $w = 0$, $||w||\le 1$ or $||w|| > 1$, it holds that, by choosing $\alpha$ satisfying:
\begin{equation}
    \label{eq:alpha_bound_in_the_proof}
    0 < \alpha < 2 {\lambda_{min}(G(u)) \over M_\delta \cdot M_{\nabla} + L + \xi \sum\limits_{i=1}^l{l_i\over 2 } + ||C_i|| M_\delta }
\end{equation}
we get $V(u^+) \le V(u)$. $\hfill \qed$\\
\end{proof}
With $\alpha$ selected as in Lemma \ref{lem:Lyapunov_difference}, we can apply the Invariance Principle \cite{Lasalle1976}, which implies that for some $c \in V(\mathcal{U})$, the trajectory $\mathbf{u} = \{u_0, u_1, \dots\}$ will converge to the largest invariant of:
\begin{equation*}
    V^{-1}(c) \cap \{ u \in \mathcal{U} \ | \ V(u^+) = V(u)\}
\end{equation*}
In the above set, i.e. when $V(u^+) = V(u)$, it holds:
\begin{align*}
    0  \le -\alpha w^\top G(u) w + {L \over 2} \alpha^2 ||w||^2 + \\ - \alpha \delta(u) \nabla_y \Phi(u,y)|_{y = h(u)} w +  \\-\sum\limits_{i=1}^l (\xi - \mu_i) \max\{0, C_i h(u) -d_i\}  + \\   \xi \sum\limits_{i=1}^l\max\{0,{l_i\over 2}\alpha^2||w||^2 - \alpha \delta(u) w\} 
\end{align*}
\\
However, by choosing $\alpha$ as in Lemma \ref{lem:Lyapunov_difference}, the right-hand side of the above inequality is negative if $w \ne 0$, implying $w = 0$. By the update rule \eqref{eq:input_evolution_nominal}, it follows that, in the convergence set, $u^+ = u$ and that $C_i h(u) \le d_i$, $\forall i \in 1, \dots, l$. We can conclude the proof as in \cite{Haberle2020}. $\hfill \qed$ \\
\end{proof}






\bibliographystyle{unsrt}
\bibliography{mirko}
\onecolumn

\section{Numerical data for example in "Robustness of a feedback optimization scheme with application to bioprocess manufacturing"}
\subsection{Introduction}
In the following all the numerical details of the bioprocessing example for the above paper "Robustness of a feedback optimization scheme with application to bioprocess manufacturing", submitted for the Conference on Decision and Control 2021, are presented, for the purpose of reproducibility of the results. The metabolic network taken into exam is reported in Figure \ref{fig:second_met_net}.
\begin{figure}[h!]
    \centering
    \includegraphics[width=0.65\textwidth]{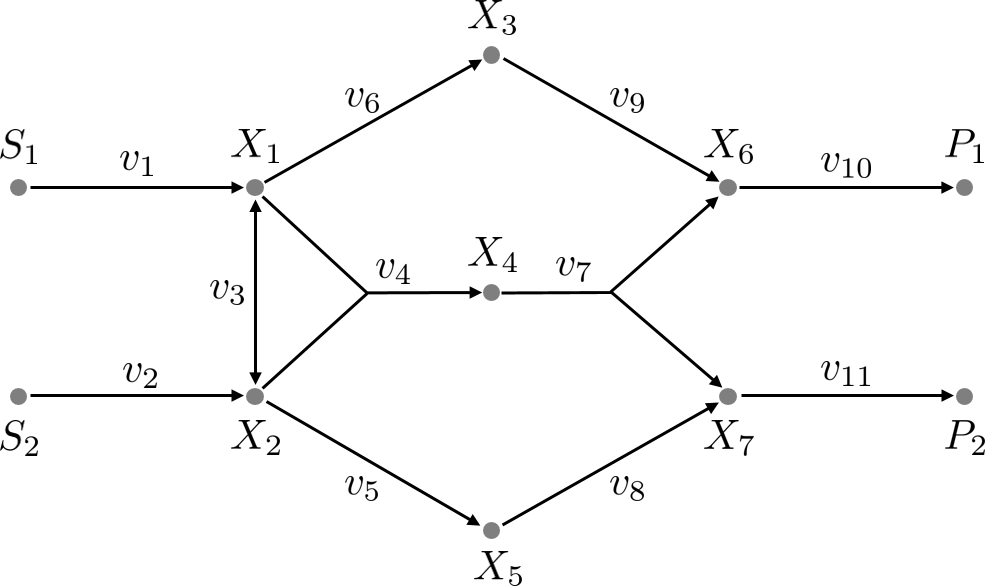}
    \caption{Example metabolic network.}
    \label{fig:second_met_net}
\end{figure}
\subsection{Numerical example data}
Mass-balance steady-state equation:
\begin{equation}
    \label{eq:mass_balance_eq}
    0 = N A_{ext} E \eta - F_1 c_{ext} + F_2 S_{in}
\end{equation}
Biomass at steady-state:
\begin{equation*}
    N = 2.15
\end{equation*}
External metabolites stoichiometric matrix:
\begin{equation}
    A_{ext} =  \begin{bmatrix}   -1&     0&     0&     0&     0&     0&     0&     0&     0&     0&     0\\
     0&    -1&     0&     0&     0&     0&     0&     0&     0&     0&     0\\
     0&     0&     0&     0&     0&     0&     0&     0&     0&     1&     0\\
     0&     0&     0&     0&     0&     0&     0&     0&     0&     0&     1 \end{bmatrix}
\end{equation}
Macro-reaction rates kinetic expressions:
\begin{align}
\label{eq:w_components}
    \eta_{1} = 1 \cdot {s_1 \over {s_1 + 1}} \cdot {1 \over {1 +  s_1}} \cdot {s_2 \over {s_2 + 2}} \cdot {1 \over {1 + s_2}}\\
    \eta_{2} = 1 \cdot {s_1 \over {s_1 + 1}} \cdot {1 \over {1 + s_1}} \cdot {s_2 \over {s_2 + 4}} \cdot {1 \over {1 +  s_2}} \nonumber\\
    \eta_{3} = 1 \cdot {s_1 \over {s_1 + 1}} \cdot {1 \over {1 + s_1}} \cdot {s_2 \over {s_2 + 0.5}} \cdot {1 \over {1 + s_2}}\nonumber\\
    \eta_{4} = 0.3 \cdot {s_1 \over {s_1 + 0.25}} \cdot {1 \over {1 + 0.25 s_1}} \cdot {s_2 \over {s_2 + 0.25}} \cdot {1 \over {1 + 0.25 s_2}}\nonumber
\end{align}
Elementary flux modes matrix:
\begin{equation}
    E = \begin{bmatrix} e_1 &e_2 &e_3 &e_4 \end{bmatrix} =  \begin{bmatrix} 0&     1&     2&     1\\
     2&     0&     0&     0\\
     0&     0&     0&     1\\
     1&     0&     1&     0\\
     0&     1&     0&     0\\
     0&     0&     0&     1\\
     1&     0&     1&     0\\
     0&    1&     0&     0\\
     1&     0&     1&     1\\
     1&     1&     1&     0\\
    -1&     1&     1&     0 \end{bmatrix}
\end{equation}
\begin{remark}
By computing the EFM matrix $E$ using METATOOL \cite{METATOOL} for the network in Figure \ref{fig:second_met_net} one will obtain $7$ columns instead of $4$. This is not an error, as we selected for this particular examples the columns $e_2$, $e_3$, $e_4$ and $e_5$ of the original matrix, by imposing $\eta_1 = \eta_6 = \eta_7 = 0$. $\eta_2$, $\eta_3$, $\eta_4$ and $\eta_5$ are respectively renamed  as $\eta_1$, $\eta_2$, $\eta_3$ and $\eta_4$, while $e_2$, $e_3$, $e_4$ and $e_5$ are respectively renamed as $e_1$, $e_2$, $e_3$ and $e_4$. $\hfill \blacktriangleleft$
\end{remark}
\noindent Input and output constraints:
\begin{align}
    0 \le S_{in}^1 \le 100 \\
    0 \le S_{in}^2 \le 100 \\
    0 \le S^1 \le 100 \\
    0 \le S^2 \le 100 \\ 
    0 \le P^1 \le 100 \\
    0 \le P^2 \le 100 
\end{align}
where $S_{in}^i$ is the concentration of the $i$-th substrate in the feed-flow, $S^i$ is the concentration of the $i$-th substrate at steady state in the bioreactor, while $P^i$ is the concentration of the $i$-th product at steady state in the bioreactor.\\ \\
Perfusion matrices:
\begin{align}
\label{eq:F_matrices}
    F_1 &= \begin{bmatrix} F &0 &0 &0 \\ 0 &F &0 &0\\ 0 &0 &F &0\\ 0 &0 &0 &F \end{bmatrix}\\
    F_2 &= \begin{bmatrix} F &0 \\ 0 &F\\ 0 &0\\ 0 &0  \end{bmatrix}\\
    F &= 0.5
\end{align}\\
\subsection{Simulations}
\subsubsection{Perturbed parameters knowledge}
For the simulations in Figure \ref{fig:pert_parameters} six realizations of random perturbations, in a $50\%$ range, on all the kinetic parameters in \eqref{eq:w_components} are considered. These realizations are obtained by uniformly sampling any parameter $\theta$  in the range $[0.5 \bar{\theta}, 1.5 \bar{\theta}]$, where $\bar{\theta}$ is the nominal value of the considered parameter. Each dashed line represents the output evolution when one of these realization, and the associated perturbed parameters are assumed to be the ones available for the model, while the real parameters are the nominal ones. The nominal parameters are the ones presented in \eqref{eq:w_components}.\\
\begin{figure}[h!]
    \centering
    \includegraphics[width=0.75\textwidth]{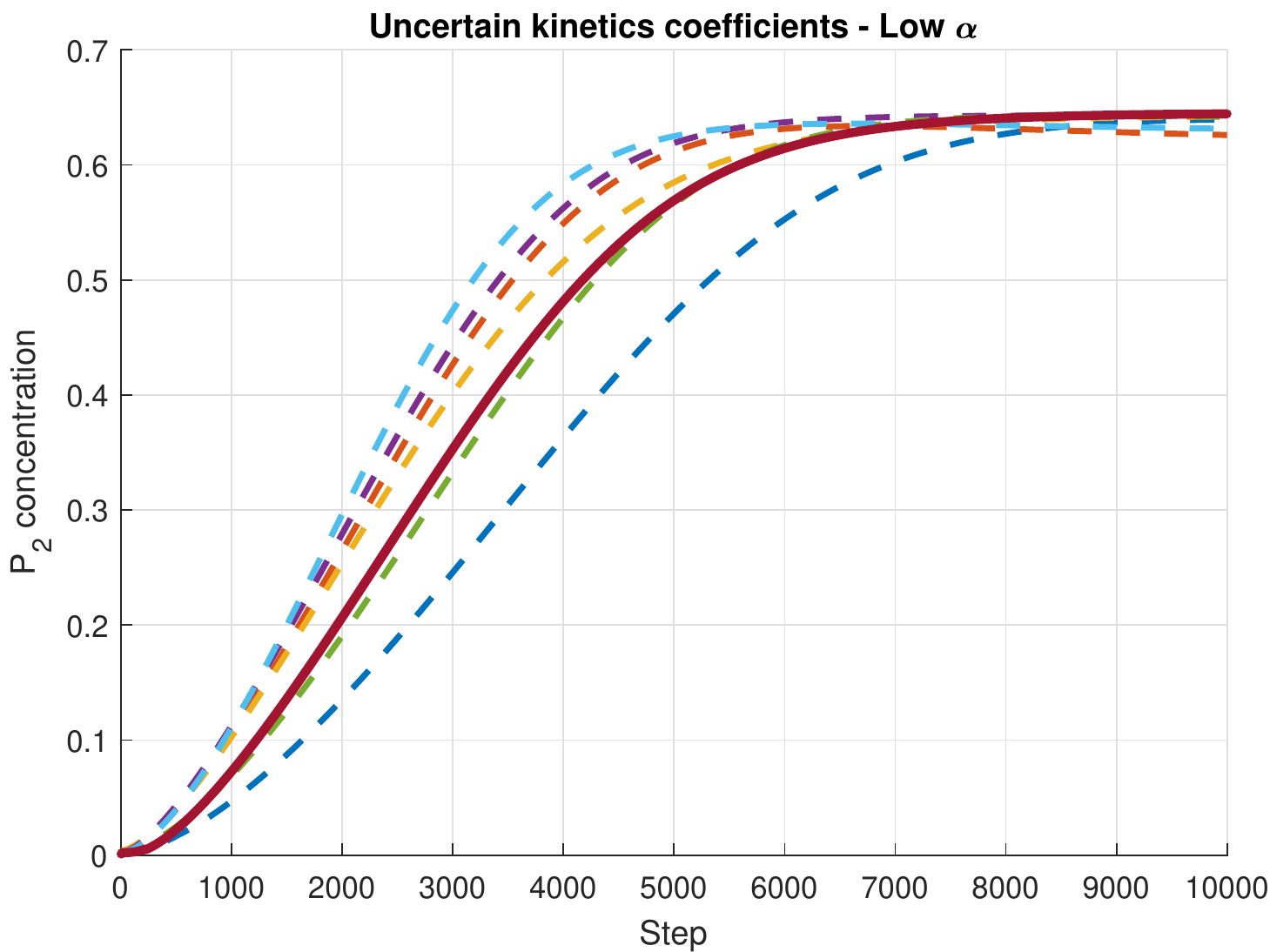}
    \caption{$P_2$ concentration evolution, with a low $\alpha$, for perturbed kinetics parameters. The solid line is the curve obtained in the nominal case, at the chosen $\alpha$, while the dashed lines are the curves obtained in six different realizations of the uncertainty.}
    \label{fig:pert_parameters}
\end{figure}
\subsubsection{Missing kinetics}
For the simulations in Figure \ref{fig:miss_kin}, the kinetics model available will be missing some components, and in order to implement this we removed one or two columns from the matrix $E$ and also removed the associated components of the vector $w$. In particular the following cases are considered:
\begin{itemize}
    \item Nominal case (solid blue line);
    \item $e_2$ and $\eta_2$ are missing (purple dashed);\\
    \item $e_3$ and $\eta_3$ are missing (yellow dashed);\\
    \item $e_4$ and $\eta_4$ are missing (red dashed);\\
    \item $e_2$, $e_3$ and $\eta_2$, $\eta_3$ are missing (light blue dashed);\\
    \item $e_3$, $e_4$ and $\eta_3$, $\eta_4$ are missing (green dashed);
\end{itemize}
\begin{figure}[h!]
    \centering
    \includegraphics[width=0.75\textwidth]{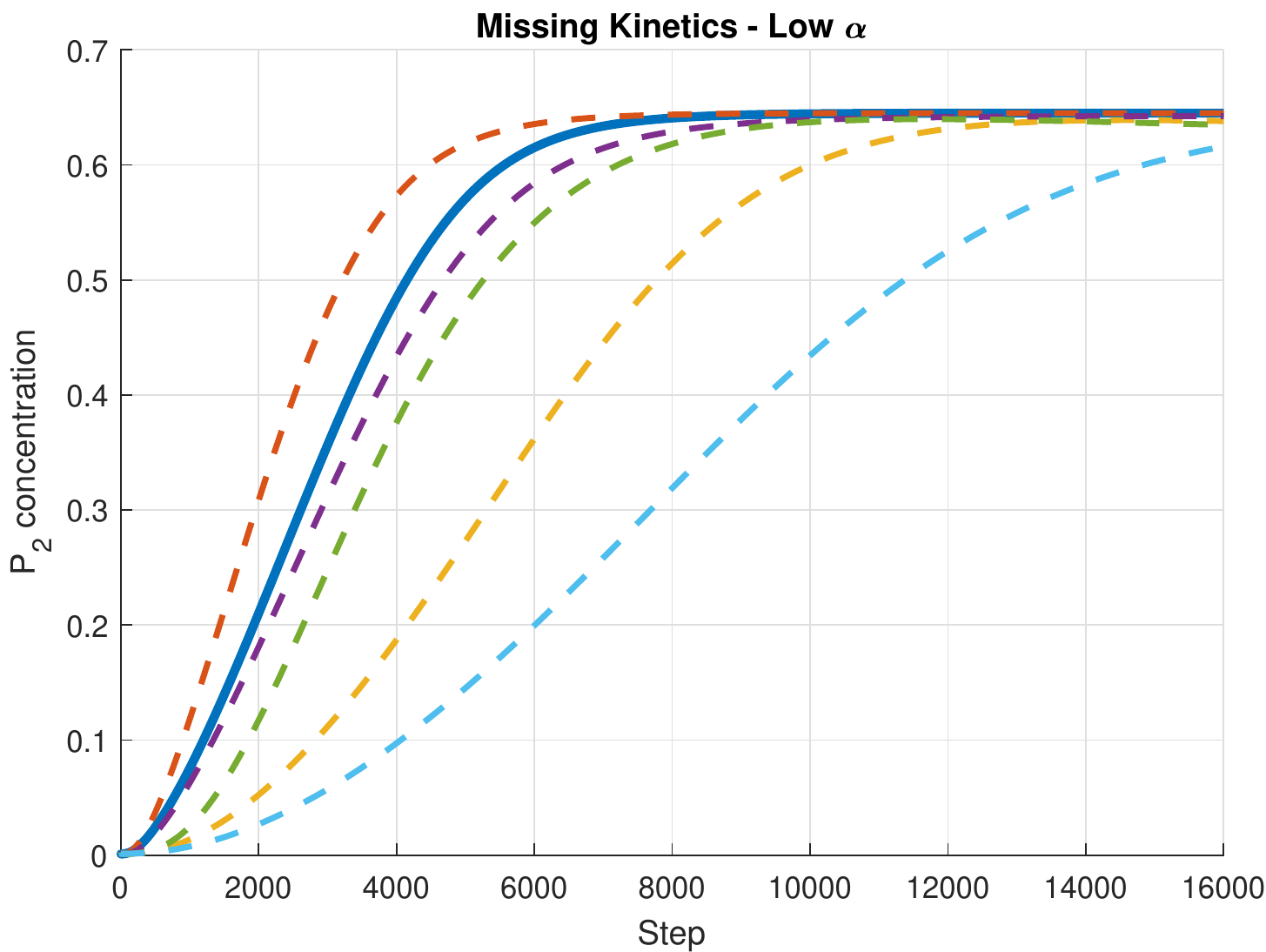}
    \caption{$P_2$ concentration evolution, with a low $\alpha$, when parts of the kinetics are not modelled. The solid line is the curve obtained in the nominal case, at the chosen $\alpha$, while the dashed lines are the curves obtained in the five cases where one or two columns of $E$ are removed.}
    \label{fig:miss_kin}
\end{figure}
\newpage    
\end{document}